\documentclass[11pt]{article}
\usepackage[a4paper,margin=26mm]{geometry}
\usepackage[T1]{fontenc}
\usepackage[utf8]{inputenc}
\usepackage{lmodern}
\usepackage{amsmath,amssymb,amsthm,mathtools}
\usepackage{booktabs,longtable,array,tabularx}
\usepackage{enumitem}
\usepackage{xcolor}
\usepackage{hyperref}
\hypersetup{colorlinks=true,linkcolor=blue!70!black,urlcolor=blue!70!black,citecolor=blue!70!black}

\newtheorem{theorem}{Theorem}
\newtheorem{lemma}{Lemma}

\theoremstyle{definition}

\newtheorem{remark}{Remark}

\newcommand{\Zfive}{\mathbb Z_5}

\newcommand{\one}{\mathbf 1}

\newcommand{\Cay}{\operatorname{Cay}}

\title{Hamilton decompositions of the directed 5-torus for odd modulus}
\author{Sanghyun Park\\Yonsei University}
\date{April 2026}

\begin{document}
\maketitle

\begin{abstract}
We prove that the directed five-dimensional torus
\[
D_5(m)=\Cay\bigl((\mathbb Z_m)^5,\{e_0,e_1,e_2,e_3,e_4\}\bigr)
\]
has a Hamilton decomposition for every odd integer \(m\ge3\).  The construction assigns the five
outgoing generators by a cyclic layer schedule with one non-constant layer.  This layer is determined
by a zero-set Latin table, and an explicit finite exact-cover certificate proves that it is a matching.
By cyclic symmetry, Hamiltonicity of all color classes is reduced to a single normalized return map.
For \(m\ge5\), an explicit first-return calculation on the section \(p=2\) gives one induced cycle
with total excursion length \(m^4\).  The remaining modulus \(m=3\) is settled by a printed finite
cycle certificate.  A separate Lean 4 formalization verifies the final Cayley statement and the finite
certificates.
\end{abstract}

\section{Introduction}

A directed Hamilton decomposition of a regular digraph is a partition of its arc set into directed
Hamilton cycles.  For
\[
D_5(m)=\Cay\bigl((\mathbb Z_m)^5,\{e_0,e_1,e_2,e_3,e_4\}\bigr),
\]
the fifth Cartesian power of the directed cycle \(\vec C_m\), this means five directed cycles of
length \(m^5\) whose arcs are exactly the directed Cayley arcs \(x\to x+e_i\).  We prove the following
theorem.

\begin{theorem}[Main theorem]\label{thm:intro-main}
For every odd integer \(m\ge3\), the arc set of \(D_5(m)\) decomposes into five directed Hamilton
cycles.
\end{theorem}

Hamilton decompositions of products of cycles have a long undirected history.  Foregger~\cite{Foregger1978}
proved a product theorem for Hamilton decompositions of Cartesian products of cycles, extending the
classical two-cycle case.  For directed products, Trotter and Erd\H{o}s~\cite{TrotterErdos1978}
studied Hamiltonicity of products of two directed cycles, and Keating~\cite{Keating1985} determined
when such a product decomposes into two directed Hamilton cycles; for surveys see
\cite{WitteGallian1984,CurranGallian1996,AlspachBermondSotteau1990}.  Curran and
Witte~\cite{CurranWitte1985} showed that the Cartesian product of three or more nontrivial directed
cycles admits a directed Hamilton cycle.  The corresponding three-factor decomposition problem was
solved for \(d=3\) in~\cite{d3torus}, with independent constructions given by
Knuth~\cite{Knuth2026} and Aquino-Michaels~\cite{AquinoMichaels2026}.  The present paper extends the
return-map method of~\cite{d3torus} to dimension five for all odd moduli.

This is the first higher-dimensional case in which the return-map method requires a genuine zero-set
selector rather than an odometer-type low-layer correction.  The odd-modulus hypothesis is used in the
half-turn computation of the first-return map; the theorem does not assert an obstruction for even
modulus, only that a different table or splice mechanism is needed there.

The proof is a cyclic arc-coloring with one non-constant layer.  The non-constant layer is defined
by a zero-set Latin table.  A finite exact-cover certificate proves that this layer is a matching;
the remaining layers are translations.  After passing to the \(m\)-step return map on a root flat,
cyclic symmetry reduces the Hamiltonicity of the five color classes to one normalized return map
\[
G(w)=w-3q_0+q_3+q_{p(Z(w))}.
\]
For \(m\ge5\), the first return of \(G\) to the section \(p=2\) is explicit.  The induced return is a
single cycle, and the sum of its excursion lengths is \(m^4\), the size of the root flat.  The case
\(m=3\) is settled by the finite cycle certificate printed in Appendix~\ref{app:m3certificate}.

The paper proof is self-contained after the finite certificates printed in Appendices~\ref{app:matching-certificate} and~\ref{app:m3certificate}.  A companion Lean 4 formalization gives an independent machine verification; artifact details are recorded in Appendix~\ref{app:formal} and in~\cite{D5OddLean}.

\section{The root flat and return maps}\label{sec:root-return}

All color and direction indices are read in \(\Zfive=\{0,1,2,3,4\}\).  Let
\[
A_m=\left\{w=(w_0,w_1,w_2,w_3,w_4)\in(\mathbb Z_m)^5:\sum_{i=0}^4 w_i=0\right\}.
\]
For \(i=0,1,2,3\), set
\[
q_i=e_i-e_4,\qquad q_4=0.
\]
Thus \(|A_m|=m^4\).  In the original directed torus, put
\[
\sigma(x)=x_0+x_1+x_2+x_3+x_4\pmod m,
\qquad X_t=\{x:\sigma(x)=t\}.
\]
Every arc goes from \(X_t\) to \(X_{t+1}\).  We identify \(X_t\) with \(A_m\) by
\[
\iota_t:X_t\longrightarrow A_m,\qquad \iota_t(x)=x-t e_4.
\]
Thus \(\sum_i \iota_t(x)_i=0\).  In this coordinate system, using direction \(e_i\) in the torus adds
\[
\iota_{t+1}(x+e_i)-\iota_t(x)=e_i-e_4=q_i
\]
in \(A_m\).

A colored layer map for color \(c\) is denoted
\[
P_{t,c}:A_m\to A_m.
\]
The \(m\)-step return of color \(c\) is
\[
R_c=P_{m-1,c}\cdots P_{1,c}P_{0,c}:A_m\to A_m.
\]

\begin{lemma}[Return criterion]\label{lem:return}
Assume the layer maps \(P_{t,c}:A_m\to A_m\) are bijections.  Then the color class \(c\) is a directed
Hamilton cycle of \(D_5(m)\) if and only if \(R_c\) is a single \(m^4\)-cycle on \(A_m\).
\end{lemma}

\begin{proof}
For a fixed color, every vertex has one outgoing edge and, by bijectivity of the layer maps, one
incoming edge.  Thus the color class is a disjoint union of directed cycles.  Since every colored edge
increases \(\sigma\) by one, observing an orbit every \(m\) steps gives precisely the cycle decomposition
of \(R_c\) on \(X_0\cong A_m\).  A return cycle of length \(\ell\) lifts to a directed cycle of length
\(m\ell\).  Therefore the color class is Hamiltonian exactly when \(R_c\) is one cycle of length
\(m^4\).
\end{proof}

\section{The zero-set Latin table}\label{sec:zero-set-table}

For \(w\in A_m\), define its zero-set
\[
Z(w)=\{i\in\Zfive:w_i=0\}.
\]
The layer-1 table uses
\[
S(w)=Z(w)-1=\{i-1:i\in Z(w)\}\subseteq\Zfive.
\]
Let \(\tau_s(c)=c+s\pmod 5\).  The table \(\Lambda_1(S)\in S_5\) is specified on representatives and
extended by cyclic equivariance:
\[
\Lambda_1(S+k)(a+k)=\Lambda_1(S)(a)+k\pmod 5.
\]
In the row notation \((p_0,p_1,p_2,p_3,p_4)\), the entry means \(\Lambda_1(S)(c)=p_c\).

\begin{table}[h]
\centering
\caption{Representative rows of the layer-1 zero-set Latin table \(\Lambda_1\).}\label{tab:lambda}
\begin{tabular}{c c}
\toprule
\(S\) & \(\Lambda_1(S)\) \\
\midrule
\(\varnothing\) & \((0,1,2,3,4)\) \\
\(\{0\}\) & \((0,1,3,2,4)\) \\
\(\{0,1\}\) & \((4,1,3,2,0)\) \\
\(\{0,2\}\) & \((4,1,3,0,2)\) \\
\(\{0,1,2\}\) & \((1,0,3,4,2)\) \\
\(\{0,1,3\}\) & \((4,3,0,2,1)\) \\
\(\{0,1,2,3,4\}\) & \((0,1,2,3,4)\) \\
\bottomrule
\end{tabular}
\end{table}

The cyclic extension of Table~\ref{tab:lambda} covers all root-flat zero-sets; by the root-flat
relation, the size-four case is vacuous, and the size-five case is the single point \(w=0\).  For
color 0 set
\[
p(Z)=\Lambda_1(Z-1)(0).
\]
The matching statement for this selector is the exact-cover certificate of Lemma~\ref{lem:exactcover}.

\section{Coloring schedules}\label{sec:schedules}

Let \(d_t(w,c)\in\Zfive\) be the direction used by color \(c\) at layer \(t\) and root point \(w\).
The layer map is
\[
P_{t,c}(w)=w+q_{d_t(w,c)}.
\]
The non-constant layer is \(t=1\).

For odd \(m\ge 5\), set
\[
\begin{aligned}
&d_0(w,c)=c,\\
&d_1(w,c)=\Lambda_1(S(w))(c),\\
&d_2(w,c)=c+3,\\
&d_3(w,c)=c+4,\\
&d_t(w,c)=c\qquad(4\le t\le m-1).
\end{aligned}
\tag{$\mathrm{Sch}_{\ge 5}$}
\]
For \(m=3\), set instead
\[
 d_0(w,c)=c+4,
\qquad d_1(w,c)=\Lambda_1(S(w))(c),
\qquad d_2(w,c)=c+3.
\tag{$\mathrm{Sch}_3$}
\]
In either schedule, every row \(c\mapsto d_t(w,c)\) is a permutation of \(\Zfive\).  Thus the outgoing
arcs at every vertex are partitioned by the five colors.

For later reference, the color-\(c\) arc set is explicitly
\[
E_c=\left\{\bigl(x,\ x+e_{d_{\sigma(x)}(\iota_{\sigma(x)}x,c)}\bigr):
x\in(\mathbb Z_m)^5\right\}.
\tag{$E_c$}
\]
For each fixed vertex \(x\), the five directions
\[
d_{\sigma(x)}(\iota_{\sigma(x)}x,c),\qquad c\in\Zfive,
\]
are exactly \(\Zfive\).  Hence the sets \(E_0,\ldots,E_4\) partition the directed Cayley arc set.
After Lemma~\ref{lem:matching}, each \(E_c\) is also a directed 1-factor, because the corresponding
layer maps are bijections.

\subsection{Cyclic equivariance on the root flat}\label{subsec:cyclic-equivariance}

By cyclic equivariance, it suffices to analyze color \(0\).

Let \(\rho_c:A_m\to A_m\) be the coordinate rotation
\[
(\rho_c w)_j=w_{j-c}\qquad (j\in\mathbb Z_5),
\]
so that zero-sets rotate by \(Z(\rho_c w)=Z(w)+c\).  The root-flat basis is not invariant term-by-term,
because the coordinate \(4\) is used in the identification with \(A_m\).  The precise correction is
\[
\rho_c(q_i)=q_{i+c}-q_{4+c}. \tag{$\mathrm{Rot}$}
\]
Thus the color transfer is affine, not a bare coordinate rotation.
For the non-constant layer put
\[
P_c(w)=w+q_{\Lambda_1(Z(w)-1)(c)} .
\]
The cyclic rule defining \(\Lambda_1\) gives
\[
\Lambda_1(Z(\rho_c w)-1)(c)=\Lambda_1((Z(w)-1)+c)(0+c)
       =\Lambda_1(Z(w)-1)(0)+c.
\]
Equivalently, for \(p=p(Z(w))\),
\[
P_c\rho_c(w)=\rho_cw+q_{p+c}
      =T_{q_{4+c}}\rho_c(w+q_p)=T_{q_{4+c}}\rho_cP_0(w).
\]
Thus
\[
P_c\rho_c=T_{q_{4+c}}\rho_c P_0. \tag{$P_c$}
\]
In particular, \(P_c\) is a bijection whenever \(P_0\) is.  The constant layers satisfy the same
coordinate-color equivariance directly, since adding \(q_i\) before applying \(\rho_c\) is the same as
adding \(q_{i+c}\) after applying \(\rho_c\), up to the layer-identification correction in
\textup{(\(\mathrm{Rot}\))}.

\section{Layer matching}\label{sec:matching-certificate}

All constant layers are translations.  It remains to prove bijectivity of the non-constant layer.  By
the cyclic equivariance relation \textup{(\(P_c\))}, it is enough to consider color \(0\):
\[
P(w)=P_0(w)=w+q_{p(Z(w))}.
\]

For \(i\in\Zfive\) and \(Z\subseteq\Zfive\), put
\[
\mathcal C_{Z,i}=\{y\in A_m:Z(y-q_i)=Z\}.
\]
The relevant finite statement is
\[
\#\{i\in\Zfive:p(Z(y-q_i))=i\}=1
\qquad(y\in A_m).
\tag{MC}
\]

\begin{lemma}[Exact-cover certificate]\label{lem:exactcover}
For every odd \(m\ge3\), the condition \textup{(MC)} holds.
\end{lemma}

\begin{proof}
Appendix~\ref{app:matching-certificate} expands the selector \(p(Z)\) on all 27 feasible zero-sets
and lists the corresponding image cells.  Lemma~\ref{lem:exactcoverappendix} is exactly the finite
assertion \textup{(MC)}.
\end{proof}

\begin{lemma}[Matching table lemma]\label{lem:matching}
For every odd \(m\ge3\), the map \(P:A_m\to A_m\), \(P(w)=w+q_{p(Z(w))}\), is a bijection.
Consequently every layer map \(P_{t,c}\) in the schedules above is a bijection.
\end{lemma}

\begin{proof}
Let \(y\in A_m\).  If \(P(w)=y\), then \(w=y-q_i\) for \(i=p(Z(w))\).  Thus possible predecessors are
the five points \(y-q_i\), and \(y-q_i\) is a predecessor precisely when \(p(Z(y-q_i))=i\).
Lemma~\ref{lem:exactcover} gives exactly one such \(i\).  Hence \(P=P_0\) is bijective.  The affine
equivariance relation \textup{(\(P_c\))} gives bijectivity of every non-constant color layer, and the
other layers are translations.
\end{proof}

\section{Normalizing the return map}\label{sec:normalizing}

Write \(T_i(w)=w+q_i\).  For color \(c\), write
\[
P_c(w)=w+q_{\Lambda_1(Z(w)-1)(c)}.
\]
For odd \(m\ge5\), the color-\(c\) return under \((\mathrm{Sch}_{\ge5})\) is
\[
R_c=T_c^{m-4}T_{c+4}T_{c+3}P_cT_c
    =T_{-4q_c+q_{c+3}+q_{c+4}}P_cT_c.
\]
Define the normalized return
\[
G_c=T_cR_cT_c^{-1}=T_{-3q_c+q_{c+3}+q_{c+4}}P_c. \tag{$G_c$}
\]
For \(m=3\), the schedule \((\mathrm{Sch}_3)\) gives
\[
R_c=T_{c+3}P_cT_{c+4},
\]
and conjugation by \(T_{c+4}\) gives the same displayed formula \textup{(\(G_c\))}, because
\(-3q_c=0\) in \(A_3\).  Thus \(G_c\) is conjugate to \(R_c\) for every odd \(m\ge3\).

For \(c=0\), \(P_0=P\) and \textup{(\(G_c\))} becomes
\[
\boxed{G(w)=G_0(w)=w-3q_0+q_3+q_{p(Z(w))}.}
\tag{$G$}
\]
Equation~\textup{(\(\mathrm{CG}\))} is an immediate consequence of the correction term in
\textup{(\(\mathrm{Rot}\))}.  Namely, \textup{(\(P_c\))} and \textup{(\(\mathrm{Rot}\))} imply
\[
G_c\rho_c=\rho_cG_0. \tag{$\mathrm{CG}$}
\]
Hence every normalized color return \(G_c\) is coordinate-rotation conjugate to \(G=G_0\).

In coordinates, for \(p=p(Z(w))\),
\[
G(w)=w+B+e_p,
\qquad B=(-3,0,0,1,1).
\tag{$G_1$}
\]
Explicitly,
\[
\Delta w_0=-3+\one_{p=0},\quad
\Delta w_1=\one_{p=1},\quad
\Delta w_2=\one_{p=2},
\]
\[
\Delta w_3=1+\one_{p=3},\quad
\Delta w_4=1+\one_{p=4}.
\tag{$G_2$}
\]

\section{The \texorpdfstring{\(p=2\)}{p=2} section and first return}\label{sec:first-return}

From Table~\ref{tab:lambda}, \(p(Z)=2\) occurs exactly for
\[
Z=\{0,3\},\qquad Z=\{0,1,3\},\qquad Z=\{0,2,3\}.
\]
Using the root-flat relation, this is equivalent to
\[
p(Z(w))=2\iff w_0=0,\quad w_3=0,\quad w_4\ne0.
\]
Thus the \(p=2\) section is
\[
\Sigma=\{w\in A_m:p(Z(w))=2\}
=
\{w(a,b)=(0,a,b,0,-a-b):a+b\ne0\}.
\]
Hence \(|\Sigma|=m(m-1)\).  In the rest of Sections~\ref{sec:first-return} and
\ref{sec:induced-cycle}, assume \(m\ge5\) and put \(m=2h+1\), so \(h=(m-1)/2\ge2\).

For \(w(a,b)\in\Sigma\), let \(\ell(a,b)\) be the first positive return time to \(\Sigma\) under \(G\),
and let \(\Phi(a,b)=(a',b')\) be the corresponding first-return point.

\begin{lemma}[First-return table]\label{lem:firstreturn}
Assume \(m\ge5\).  Let all coordinates be read in \(\mathbb Z_m\), and write
\(s=a+b\in\{1,\ldots,2h\}\).  The following first-return table holds.

\medskip\noindent
\textbf{Case I: \(0\le b\le m-2\).}
\[
b'=b+1,
\qquad
 a'=\begin{cases}
 a,&s=h,\\
 a+h,&s\ne h.
 \end{cases}
\]
The return time is
\[
\ell(a,b)=
\begin{cases}
(h+1)m, & 1\le s\le h-1,\\
2(h+1)m, & s=h,\\
(3h+2)m, & h+1\le s\le 2h.
\end{cases}
\]

\medskip\noindent
\textbf{Case II: \(b=m-1\).}
Here \(a=1\) is excluded because \(a+b=0\).  Moreover
\[
\Phi(0,m-1)=(1,0),
\qquad
\Phi(a,m-1)=(a,0)\quad(a\ne0,1),
\]
and
\[
\ell(0,m-1)=m^3-(m-1)(m-2),
\qquad
\ell(a,m-1)=m-1\quad(a\ne0,1).
\]
\end{lemma}

\subsection{Block recurrence and case analysis}

We now prove the first-return table from the displayed zero-set table.  A power such as \(0^r\)
denotes a repeated direction word and is empty when \(r=0\).  Since \(G^m\) is the state after one full
layer-cycle of color 0, the following lemma records the first such block from a normal row.

\begin{lemma}[The first block from a normal row]\label{lem:firstblock}
Assume \(0\le b\le m-2\), put \(s=a+b\ne0\), and set \(B=b+1\).  Then
\[
G^m w(a,b)=(-2,\ a+1,\ B,\ 0,\ -s).
\]
\end{lemma}

\begin{proof}
At \(w(a,b)\) the zero-set contains \(0\) and \(3\), and \(w_4=-s\ne0\), so the table gives \(p=2\).
After this first step the fifth coordinate is \(1-s\), while \(w_0\ne0\), \(w_2=b+1\ne0\), and
\(w_3=1\).  During the following run before \(w_4\) becomes zero, we have
\(w_2\ne0\), \(w_3\ne0\), and \(w_4\ne0\).  The only coordinates that may vanish are \(w_0\) and
\(w_1\).  Hence the possible zero-sets are
\[
\varnothing,\qquad \{0\},\qquad \{1\},\qquad \{0,1\},
\]
and Table~\ref{tab:lambda} gives \(p=0\) in all four cases.  Thus the selector remains
\(p=0\) until the fifth coordinate first becomes zero, which occurs after \(s-1\) such steps.  At that
moment \(w_2=b+1\) and \(w_3=s\) are both nonzero, so the possible zero-sets are only
\(\{4\}\), \(\{0,4\}\), \(\{1,4\}\), and \(\{0,1,4\}\); the cyclic extension of Table~\ref{tab:lambda} gives
\(p=1\) in all four rows.  After the
\(p=1\) step the fifth coordinate is \(1\), and the remaining \(m-s-1\) steps again have selector
\(p=0\).  Hence the direction word is
\[
2,\quad 0^{s-1},\quad 1,\quad 0^{m-s-1}.
\]
Thus, in this block,
\[
(N_0,N_1,N_2,N_3,N_4)=(m-2,1,1,0,0).
\]
Using the displacement formula
\[
(-3\ell+N_0,\ N_1,\ N_2,\ \ell+N_3,\ \ell+N_4),\qquad \ell=m,
\]
gives the displayed state.
\end{proof}

We next study only the states that occur at multiples of \(m\) after this first block.  Write such a
state as
\[
Y=(x,y,B,0,z),\qquad B\ne0.
\]
As long as \(Y\notin\Sigma\), the next full block is governed by a two-coordinate map on \((x,z)
\in\mathbb Z_m^2\).

\begin{lemma}[Normal-row block recurrence]\label{lem:blockrecurrence}
Let \(Y=(x,y,B,0,z)\), with \(B\ne0\), be outside \(\Sigma\).  Then
\[
G^mY=(x',y',B,0,z')
\]
where
\[
(x',z')=
\Theta(x,z):=
\begin{cases}
(x-1,0),& z=-1,\\
(-1,0),& x=0,\ z=0,\\
(x-2,z+1),&\text{otherwise.}
\end{cases}
\tag{$\Theta$}
\]
Moreover the \(y\)-coordinate is unchanged in the first case and increases by \(1\) in the other two
cases.
\end{lemma}

\begin{proof}
This is a direct substitution into Table~\ref{tab:lambda}.  The initial step of a block is governed by
the following three alternatives:
\[
\begin{array}{c|c|c|c}
\text{condition} & \text{possible initial zero-sets} & \text{block word} & (x,z)\text{-effect}\\
\hline
z=-1 & \{3\},\{1,3\} & 4,0^{m-1} & (x-1,0)\\
x=0,\ z=0 & \{0,3,4\} & 1,0^{m-1} & (-1,0)\\
\text{otherwise} & \text{remaining outside-}\Sigma\text{ cases} & 4,0^{r},1,0^{m-r-2} & (x-2,z+1)
\end{array}
\]
Here \(r\) is the unique time at which the fifth coordinate becomes zero in the generic flight.
In the middle line the zero-set is exactly \(\{0,3,4\}\): if \(x=z=0\), the root-flat relation gives
\(y=-B\ne0\).
If \(z=-1\), then \(w_4\ne0\).  Since
\(Y\notin\Sigma\), we also have \(x\ne0\).  With \(B\ne0\), the only possible zero-sets are
\(\{3\}\) and \(\{1,3\}\), according as \(y\ne0\) or \(y=0\).  Table~\ref{tab:lambda} gives \(p=4\) in
both cases.  After that step the fifth coordinate is \(1\), while \(w_3\ne0\).  Before each of the
following \(m-1\) updates the state lies in a row whose selector is \(p=0\), and after these updates
the next block boundary has fifth coordinate \(0\).  Thus the next block has direction
word
\[
4,0^{m-1},
\]
so \(x\) changes by \(-1\), \(z\) changes from \(-1\) to \(0\), and no \(p=1\) step occurs.

If \((x,z)=(0,0)\), then the root-flat relation gives \(y=-B\ne0\), so
\(Z(Y)=\{0,3,4\}\), and the table gives \(p=1\).  The following \(m-1\)
states again stay in the \(p=0\) rows until the block boundary, so the word is
\[
1,0^{m-1},
\]
so \((x,z)\) becomes \((-1,0)\) and the \(y\)-coordinate increases by \(1\).

In all remaining cases, the first step has \(p=4\).  During the following flight the fifth coordinate
advances by one at each step, so exactly once it reaches the zero-set rows with selector \(p=1\); all
other steps have selector \(p=0\).  Hence the word has counts
\[
N_4=1,\qquad N_1=1,\qquad N_0=m-2,
\]
and this gives \((x,z)\mapsto(x-2,z+1)\), with \(y\mapsto y+1\).  These three alternatives are
mutually exclusive and cover all states of the displayed form outside \(\Sigma\).
\end{proof}

\begin{lemma}[Solving the normal-row recurrence]\label{lem:normalrowsolved}
For \(0\le b\le m-2\), Lemma~\ref{lem:firstreturn} holds.
\end{lemma}

\begin{proof}
By Lemma~\ref{lem:firstblock}, after the first \(m\)-block we are at
\[
Y_1=(-2,a+1,b+1,0,-s).
\]
Thus the block recurrence starts from
\[
(x_1,z_1)=(-2,-s).
\]
The return condition is \(x=0\) and \(z\ne0\), because \(B=b+1\ne0\) and the third coordinate in
\(\Sigma\) is unrestricted.

\medskip
\noindent\textbf{Case \(s\in\{1,\ldots,h-1\}\) (short generic).}
First suppose \(1\le s\le h-1\).  Under the generic branch of \(\Theta\), the value \(z=-1\) is reached
after \(s-1\) applications, at
\[
(-2s,-1).
\]
The seam branch sends this to \((-2s-1,0)\).  After a further \(h-s\) generic applications, the
\(x\)-coordinate is \(0\), while \(z=h-s\ne0\).  Hence the first return is at block index \(h+1\), so
\[
\ell=(h+1)m.
\]
The \(y\)-coordinate starts at \(a+1\) and is increased in exactly \(h-1\) of these blocks; hence
\[
y=a+h.
\]
Therefore \(\Phi(a,b)=(a+h,b+1)\).

\medskip
\noindent\textbf{Case \(s=h\) (double seam).}
Now suppose \(s=h\).  The same calculation reaches \((0,0)\).  This is not in \(\Sigma\), and the
seam branch of \(\Theta\) sends it to \((-1,0)\).  A further \(h\) generic blocks are needed to reach
\((0,h)\).  Thus the first return is at block index \(2(h+1)\), and
\[
\ell=2(h+1)m.
\]
The total increase of \(y\) after the first block is \(2h\), so the final \(y\)-coordinate is
\(a+1+2h\equiv a\).  Hence \(\Phi(a,b)=(a,b+1)\).

\medskip
\noindent\textbf{Case \(s\in\{h+1,\ldots,2h\}\) (wrap).}
Finally suppose \(h+1\le s\le2h\).  Again \(z=-1\) is reached after \(s-1\) generic applications,
then the seam branch is taken.  From the resulting state \((-2s-1,0)\), the first solution of
\(-2s-1-2r=0\) in \(\mathbb Z_m\) with \(r\ge0\) is
\[
r=m+h-s.
\]
This value is positive and gives \(z=r\ne0\).  Hence the number of applications of \(\Theta\) after
\(Y_1\) is
\[
s+(m+h-s)=m+h=3h+1,
\]
so the return is at block index \(3h+2\), and
\[
\ell=(3h+2)m.
\]
The \(y\)-coordinate increases by
\[
(s-1)+(m+h-s)=m+h-1=3h
\]
after the first block, so the final \(y\)-coordinate is
\[
a+1+3h\equiv a+h.
\]
Therefore \(\Phi(a,b)=(a+h,b+1)\).

In all three cases, the displayed sequence of \(\Theta\)-states reaches \(x=0\) with \(z\ne0\) for the
first time at the stated block index; before that, either \(x\ne0\), or \(x=0\) occurs only together
with \(z=0\).  Hence no earlier block boundary lies in \(\Sigma\).  Lemma~\ref{lem:blockrecurrence}
shows that within each non-final block \(w_3\ne0\) except at the boundary, so no interior point of such
a block lies in \(\Sigma\).  Thus the returns are first returns.
\end{proof}

\begin{lemma}[Short last-row excursions]\label{lem:shortlastrow}
Let \(b=m-1=-1\) and \(a\ne0,1\).  We use the representative \(a\in\{2,\ldots,m-1\}\), so
\(a-2\ge0\).  Then
\[
\Phi(a,m-1)=(a,0),\qquad \ell(a,m-1)=m-1.
\]
\end{lemma}

\begin{proof}
Starting from
\[
w(a,-1)=(0,a,-1,0,1-a),
\]
the zero-set table gives the direction word of length \(m-1\)
\[
2,\quad 0^{a-2},\quad 3,\quad 0^{m-a-1}.
\]
Thus the counts are
\[
(N_0,N_1,N_2,N_3,N_4)=(m-3,0,1,1,0),\qquad \ell=m-1.
\]
Substituting into the displacement formula gives
\[
(0,a,0,0,-a)=w(a,0).
\]
The same explicit word has no intermediate point with \(w_0=0,w_3=0,w_4\ne0\), so this is the first
return.
\end{proof}

The only section point not covered by Lemmas~\ref{lem:normalrowsolved} and \ref{lem:shortlastrow} is
\[
x_*=(0,m-1).
\]

\begin{lemma}[Long wrap skeleton]\label{lem:longwrapclosing}
The remaining point satisfies
\[
\Phi(0,m-1)=(1,0),\qquad
\ell(0,m-1)=m^3-(m-1)(m-2).
\]
\end{lemma}

\begin{proof}
The orbit from \(x_*=w(0,m-1)\) first passes through the all-zero point and then enters the family
\[
E(u,v)=(u,v,0,0,-u-v),
\]
with \(u\ne0\) and \(u+v\ne0\).  Inside this family three transitions appear: a quick seam step when
\(v+1=0\), a length-\((m-1)\) generic step when neither seam is hit, and a long seam step of length
\(3m-2\) when \(v+1=-u\).  The formulas \textup{(\(LW_1\))}--\textup{(\(LW_3\))} below give these three moves
explicitly.

At every such point the table gives \(p=1\), since the zero-set
contains \(\{2,3\}\), does not contain \(4\), and does not contain \(0\).  The long wrap starts from
\[
x_*=w(0,m-1)=(0,0,-1,0,1).
\]
The initial direction word is
\[
2,\quad 4^{m-1},\quad 0,\quad 4^{m-1}.
\]
Indeed, the first step is the section step \(p=2\).  Then \(w_1=w_2=0\) and \(w_3\ne0\), so the table
gives \(p=4\) until the all-zero point is reached after \(m-1\) steps.  At the all-zero point the table
gives \(p=0\), and then the same \(p=4\) run of length \(m-1\) ends at
\[
G^{2m}x_*=E(1,0).
\]
This initial segment has no intermediate point of \(\Sigma\): whenever \(w_3=0\) before the endpoint,
one also has \(w_4=0\).

It remains to follow the orbit from the points \(E(u,v)\).  Put \(V=v+1\).  The first step from
\(E(u,v)\) has selector \(p=1\) and gives
\[
(u-3,V,0,1,1-u-v)=(u-3,V,0,1,2-u-V).
\]
If \(V=0\), then \(w_1=0\) and \(w_3\ne0\) until the next boundary, so the next \(m-1\) selectors are
all \(4\).  Therefore
\[
G^mE(u,m-1)=E(u,0). \tag{$LW_1$}
\]

Assume now that \(V\ne0\).  Until \(w_3\) next becomes zero, the coordinate \(w_1=V\) is fixed.  Hence
the only possible nonzero selector during this flight is \(p=3\), and it occurs exactly when
\(w_4=0\); all other steps have selector \(p=0\).  Since both \(p=0\) and \(p=3\) increase \(w_4\) by
one, after the \(p=1\) step the successive values of \(w_4\) are
\[
2-u-V+j\qquad (j=0,1,2,\ldots).
\]
Thus the \(p=3\) events occur at the values \(j\equiv u+V-2\pmod m\).

If \(V\ne -u\), then the first such \(j\) lies in \(\{0,
\ldots,m-3\}\); the excluded value \(j=m-1\) would mean \(u+v=0\), which is not an \(E\)-state, and
\(j=m-2\) is exactly the case \(V=-u\).  Hence there is exactly one \(p=3\) event among the next
\(m-2\) steps.  At that time
\[
w_3=1+(m-2)+1=0,
\qquad
w_4=(2-u-V)+(m-2)=-u-V\ne0,
\]
and the first coordinate is
\[
u-3-2(m-2)-1=u.
\]
Before this endpoint, after \(k<m-2\) steps of the flight, the coordinate \(w_3\) is represented by
\(1+k+\varepsilon\) with \(\varepsilon\in\{0,1\}\); hence it lies in \(\{1,\ldots,m-1\}\) and is not zero.
Therefore
\[
G^{m-1}E(u,v)=E(u,v+1)\qquad (v+1\ne0,\; v+1\ne -u). \tag{$LW_2$}
\]

Finally suppose \(V=-u\), i.e. \(v=-u-1\).  Then \(w_4\) starts at \(2\), so the first two \(p=3\)
events occur at \(j=m-2\) and \(j=2m-2\).  The only intermediate vanishing of \(w_3\) occurs at
\(j=2m-2\), where \(w_4=0\); hence that point is neither in \(\Sigma\) nor in the \(E\)-family, and the
selector there is again \(p=3\).  After \(3m-3\) steps following the initial \(p=1\) step, there have
been exactly two \(p=3\) events, and
\[
w_3=1+(3m-3)+2=0,
\qquad
w_4=2+(3m-3)=-1,
\]
while
\[
w_0=u-3-2(3m-3)-2=u+1.
\]
Thus
\[
G^{3m-2}E(u,-u-1)=
\begin{cases}
E(u+1,-u),&u\ne m-1,\\
w(1,0),&u=m-1.
\end{cases} \tag{$LW_3$}
\]
In each of \textup{(\(LW_1\))}--\textup{(\(LW_3\))}, the displayed endpoint is the first time in the
segment with \(w_3=0\) and \(w_4\ne0\); in the long case the earlier time with \(w_3=0\) has
\(w_4=0\).

Starting from \(E(1,0)\), the rules \textup{(\(LW_1\))}--\textup{(\(LW_3\))} visit, for each fixed
\(u=1,2,\ldots,m-1\), the \(m-1\) values
\[
v=1-u,\; 2-u,\; \ldots,\; -u-1
\]
cyclically, omitting only \(v=-u\).  The last seam, from \(E(m-1,0)\), lands at \(w(1,0)\).  Hence
\(\Phi(0,m-1)=(1,0)\).

The length is the sum of the displayed segments.  The initial segment has length \(2m\).  Among the
\((m-1)^2\) points \(E(u,v)\), the case \(v+1=0\) occurs \(m-2\) times, the seam case \(v+1=-u\)
occurs \(m-1\) times, and the remaining cases occur \((m-2)^2\) times.  Therefore
\[
\ell(0,m-1)=2m+(m-2)m+(m-1)(3m-2)+(m-2)^2(m-1).
\]
A simplification gives
\[
\ell(0,m-1)=m^3-m^2+3m-2=m^3-(m-1)(m-2).
\]
\end{proof}

\begin{proof}[Proof of Lemma~\ref{lem:firstreturn}]
The normal rows are Lemma~\ref{lem:normalrowsolved}; the short last-row cases are
Lemma~\ref{lem:shortlastrow}; the remaining long wrap is Lemma~\ref{lem:longwrapclosing}.  Together
these three statements are exactly the first-return table.
\end{proof}

\section{The induced cycle and total excursion length}\label{sec:induced-cycle}

\begin{lemma}[The induced map is one cycle]\label{lem:PhiCycle}
Assume \(m\ge5\).  The first-return map \(\Phi\) is a single cycle of length \(m(m-1)\) on \(\Sigma\).
\end{lemma}

\begin{proof}
Every row of Lemma~\ref{lem:firstreturn} sends \(b\) to \(b+1\), with \(b=m-1\) returning to \(0\).
Thus \(m\) applications are needed to return to the same \(b\)-row.  It remains to compute the induced
map after one full turn through the rows.

Let \((a_b,b)\) be the point after \(b\) normal-row moves, starting from \((a_0,0)=(a,0)\), and put
\(s_b=a_b+b\).  For \(0\le b\le m-2\), the normal-row rule sends
\[
s_b\longmapsto
\begin{cases}
 s_b+h+1,&s_b\ne h,\\
 h+1,&s_b=h,
\end{cases}
\]
with values read in \(\mathbb Z_m\) and with \(0\) omitted from the section.  This is the map
``add \(h+1\) modulo \(m\), and if the result would be \(0\), replace it by \(h+1\)''.  Since
\(\gcd(h+1,2h+1)=1\), addition by \(h+1\) is one cycle on \(\mathbb Z_m\).  Removing the
zero residue and splicing its predecessor directly to its successor gives the displayed map on
\(\mathbb Z_m\setminus\{0\}\), so its orbit on the nonzero residues is one cycle.  Hence during the
\(m-1\) normal-row moves, the value \(s_b=h\) occurs exactly once.  Therefore \(a_b\) is increased by
\(h\) exactly \(m-2\) times and is unchanged once, so before the last-row move
\[
a_{m-1}=a+h(m-2)\equiv a+1\pmod m.
\]
The last-row rule fixes this value if it is nonzero and sends \(0\) to \(1\).  Thus, for every
\(a\in\{1,\ldots,m-1\}\),
\[
\Phi^m(a,0)=(a+1,0),
\]
where \(m-1+1\) is read as \(1\).  Hence \(\Phi^m\) is one cycle of length \(m-1\) on
\(\Sigma_0=\{(a,0):a\ne0\}\), and \(\Phi\) itself is one cycle of length \(m(m-1)\) on \(\Sigma\).
\end{proof}

\begin{lemma}[Total excursion length]\label{lem:total}
Assume \(m\ge5\).  Then
\[
\sum_{(a,b)\in\Sigma}\ell(a,b)=m^4.
\]
\end{lemma}

\begin{proof}
Fix a normal row \(0\le b\le m-2\).  As \((a,b)\) ranges in \(\Sigma\), the value \(s=a+b\) runs
through \(1,2,\ldots,2h\) once.  Thus the row sum is
\[
(h-1)(h+1)m+2(h+1)m+h(3h+2)m.
\]
This equals
\[
\bigl((h+1)^2+h(3h+2)\bigr)m=(4h^2+4h+1)m=(2h+1)^2m=m^3.
\]
There are \(m-1\) normal rows.  In the last row, the sum is
\[
m^3-(m-1)(m-2)+(m-2)(m-1)=m^3.
\]
Therefore the total is \((m-1)m^3+m^3=m^4\).
\end{proof}

\begin{lemma}[Induced-map splice lemma]\label{lem:splice}
Let \(F\) be a bijection of a finite set \(X\), and let \(\Sigma\subseteq X\).  Suppose every point of
\(\Sigma\) has a positive first return to \(\Sigma\).  If the first-return map on \(\Sigma\) is one cycle
and the sum of its first-return times is \(|X|\), then \(F\) is one cycle on \(X\).
\end{lemma}

\begin{proof}
Start at any point of \(\Sigma\) and concatenate the first-return excursions in the cyclic order of the
induced map.  Since \(F\) is bijective, no excursion can internally repeat before its first return, and
different excursions cannot overlap without forcing a repeat in the induced cycle.  The concatenated
orbit has length equal to the sum of all first-return times, namely \(|X|\).  Hence this one orbit
contains every point of \(X\).  In particular, there is no additional cycle disjoint from
\(\Sigma\), and \(F\) is a single cycle.
\end{proof}

\begin{lemma}[Cycle lemma for \(m\ge5\)]\label{lem:cycle}
For every odd \(m\ge5\), the normalized return map \(G\) is a single \(m^4\)-cycle on \(A_m\).
\end{lemma}

\begin{proof}
By Lemma~\ref{lem:matching}, \(P\) is bijective, hence \(G=T_{-3q_0+q_3}P\) is bijective.  Lemma~\ref{lem:firstreturn} gives positive first returns from every point of \(\Sigma\), with induced
map \(\Phi\).  Lemma~\ref{lem:PhiCycle} says that \(\Phi\) is one cycle, and Lemma~\ref{lem:total}
says that the sum of all induced excursion lengths is \(|A_m|=m^4\).  Lemma~\ref{lem:splice}
therefore applies with \(F=G\) and \(X=A_m\).
\end{proof}

\begin{lemma}[The \(m=3\) finite return certificate]\label{lem:m3certificate}
For the schedule \((\mathrm{Sch}_3)\), every color return \(R_c\) is a single cycle on \(A_3\).
\end{lemma}

\begin{proof}
It is enough to certify the normalized color-0 return \(G\).  Appendix~\ref{app:m3certificate} prints
an explicit sequence
\[
\alpha_0,\alpha_1,\ldots,\alpha_{80}\in A_3
\]
which exhausts \(A_3\) and satisfies \(G(\alpha_r)=\alpha_{r+1}\), with indices read modulo \(81\).
Therefore \(G=G_0\) is one \(81\)-cycle.  For \(m=3\), Section~\ref{sec:normalizing} shows that each
\(G_c\) is conjugate to \(R_c\), and equation \textup{(\(\mathrm{CG}\))} gives
\(G_c\rho_c=\rho_cG_0\).  Hence every \(G_c\), and therefore every \(R_c\), is a single \(81\)-cycle.
\end{proof}

\section{Putting it together}\label{sec:putting}

\begin{theorem}\label{thm:main}
For every odd \(m\ge3\), the directed torus
\[
D_5(m)=\operatorname{Cay}((\mathbb Z_m)^5,\{e_0,e_1,e_2,e_3,e_4\})
\]
has a Hamilton decomposition.
\end{theorem}

\begin{proof}
The schedules define a Latin outgoing coloring because each row \(c\mapsto d_t(w,c)\) is a permutation
of the five directions.  Lemma~\ref{lem:matching} gives the incoming condition: for every color and
every layer, the layer map is a bijection.  Hence each color class is a directed 1-factor.

Since \(m\) is odd and \(m\ge3\), either \(m=3\) or \(m\ge5\).  If \(m=3\),
Lemma~\ref{lem:m3certificate} says that every color return \(R_c\) is one \(81\)-cycle.  Now assume
\(m\ge5\).  For each color \(c\),
the normalized return \(G_c\) is conjugate to \(R_c\).  Equation~\textup{(\(\mathrm{CG}\))} gives
\(G_c\rho_c=\rho_cG_0\), and Lemma~\ref{lem:cycle} says that \(G_0=G\) is a single \(m^4\)-cycle.
Therefore every \(G_c\), and hence every \(R_c\), is a single \(m^4\)-cycle.

In both cases, Lemma~\ref{lem:return} lifts each color class to one directed Hamilton cycle of length
\(m^5\).  The explicit arc sets \(E_c\) in \textup{(\(E_c\))} partition the directed Cayley arc set, so
these five Hamilton cycles form a Hamilton decomposition.
\end{proof}

\section{Discussion}\label{sec:discussion}

The relationship with the directed \(3\)-torus result~\cite{d3torus} is structural.  In dimension
three the reduced return is an odometer or finite-defect odometer.  In dimension five the same
primitiveity mechanism is carried by the zero-set changes, the two seam branches of \(\Theta\), and
the long wrap skeleton.

The odd-modulus hypothesis enters through the half-turn parameter \(m=2h+1\) and the row-cycle
computation for \(\Phi\).  Even moduli and higher-dimensional analogues require different tables or
splice mechanisms.  Ancillary search code for such finite zero-set tables and return sections is
included in~\cite{D5OddLean}.

\appendix

\section{Expanded layer-1 matching certificate}\label{app:matching-certificate}

This appendix expands the finite certificate used in Lemma~\ref{lem:matching}.  The certificate
separates the infinite parameter \(m\) from the finite zero-pattern bookkeeping.  The only symbols
that enter the bookkeeping are
\[
 y_j=0,\qquad y_j=1,\qquad y_j=-1.
\]
For odd \(m\ge3\), these three values are distinct whenever they are different as symbols.

Table~\ref{tab:selector} records the selector values generated by the cyclic extension of Table~\ref{tab:lambda}.

\begin{table}[h]
\centering
\caption{The color-0 selector \(p(Z)=\Lambda_1(Z-1)(0)\) on all feasible root-flat zero-sets.}\label{tab:selector}
\small
\begin{tabular}{c c@{\qquad}c c@{\qquad}c c}
\toprule
\(Z\) & \(p(Z)\) & \(Z\) & \(p(Z)\) & \(Z\) & \(p(Z)\) \\
\midrule
$\varnothing$ & $0$ & $\{0\}$ & $0$ & $\{1\}$ & $0$ \\
$\{2\}$ & $0$ & $\{3\}$ & $4$ & $\{4\}$ & $1$ \\
$\{0,1\}$ & $0$ & $\{0,2\}$ & $0$ & $\{0,3\}$ & $2$ \\
$\{0,4\}$ & $1$ & $\{1,2\}$ & $4$ & $\{1,3\}$ & $4$ \\
$\{1,4\}$ & $1$ & $\{2,3\}$ & $1$ & $\{2,4\}$ & $3$ \\
$\{3,4\}$ & $4$ & $\{0,1,2\}$ & $4$ & $\{0,1,3\}$ & $2$ \\
$\{0,1,4\}$ & $1$ & $\{0,2,3\}$ & $2$ & $\{0,2,4\}$ & $3$ \\
$\{0,3,4\}$ & $1$ & $\{1,2,3\}$ & $1$ & $\{1,2,4\}$ & $4$ \\
$\{1,3,4\}$ & $4$ & $\{2,3,4\}$ & $3$ & $\{0,1,2,3,4\}$ & $0$ \\
\bottomrule
\end{tabular}
\end{table}

For a target \(y\in A_m\), a predecessor through direction \(i\) is \(w=y-q_i\).  Hence the condition
that the predecessor has zero-set \(Z\) is the signature condition displayed below.  The row contributes
to the matching exactly when \(i=p(Z)\).  Thus the rows of Table~\ref{tab:cellcertificate} are the
actual image cells
\[
\mathcal C_{Z,p(Z)}=\{y\in A_m: Z(y-q_{p(Z)})=Z\}.
\]

\begin{longtable}{c c p{0.35\textwidth} p{0.35\textwidth}}
\caption{The 27 image-cell signatures for the map \(P(w)=w+q_{p(Z(w))}\).  Equalities and inequalities are read inside \(A_m\).}\label{tab:cellcertificate}\\
\toprule
\(Z\) & \(p(Z)\) & forced equalities & forbidden equalities \\
\midrule
\endfirsthead
\multicolumn{4}{c}{\tablename\ \thetable{} -- continued from previous page}\\
\toprule
\(Z\) & \(p(Z)\) & forced equalities & forbidden equalities \\
\midrule
\endhead
\(\varnothing\) & 0 & \(\text{none}\) & \(\begin{array}{l}y_0\ne1\\ y_4\ne-1\\ y_1\ne0\\ y_2\ne0\\ y_3\ne0\end{array}\) \\
\(\{0\}\) & 0 & \(\begin{array}{l}y_0=1\end{array}\) & \(\begin{array}{l}y_4\ne-1\\ y_1\ne0\\ y_2\ne0\\ y_3\ne0\end{array}\) \\
\(\{1\}\) & 0 & \(\begin{array}{l}y_1=0\end{array}\) & \(\begin{array}{l}y_0\ne1\\ y_4\ne-1\\ y_2\ne0\\ y_3\ne0\end{array}\) \\
\(\{2\}\) & 0 & \(\begin{array}{l}y_2=0\end{array}\) & \(\begin{array}{l}y_0\ne1\\ y_4\ne-1\\ y_1\ne0\\ y_3\ne0\end{array}\) \\
\(\{3\}\) & 4 & \(\begin{array}{l}y_3=0\end{array}\) & \(\begin{array}{l}y_0\ne0\\ y_1\ne0\\ y_2\ne0\\ y_4\ne0\end{array}\) \\
\(\{4\}\) & 1 & \(\begin{array}{l}y_4=-1\end{array}\) & \(\begin{array}{l}y_1\ne1\\ y_0\ne0\\ y_2\ne0\\ y_3\ne0\end{array}\) \\
\(\{0,1\}\) & 0 & \(\begin{array}{l}y_0=1\\ y_1=0\end{array}\) & \(\begin{array}{l}y_4\ne-1\\ y_2\ne0\\ y_3\ne0\end{array}\) \\
\(\{0,2\}\) & 0 & \(\begin{array}{l}y_0=1\\ y_2=0\end{array}\) & \(\begin{array}{l}y_4\ne-1\\ y_1\ne0\\ y_3\ne0\end{array}\) \\
\(\{0,3\}\) & 2 & \(\begin{array}{l}y_0=0\\ y_3=0\end{array}\) & \(\begin{array}{l}y_2\ne1\\ y_4\ne-1\\ y_1\ne0\end{array}\) \\
\(\{0,4\}\) & 1 & \(\begin{array}{l}y_4=-1\\ y_0=0\end{array}\) & \(\begin{array}{l}y_1\ne1\\ y_2\ne0\\ y_3\ne0\end{array}\) \\
\(\{1,2\}\) & 4 & \(\begin{array}{l}y_1=0\\ y_2=0\end{array}\) & \(\begin{array}{l}y_0\ne0\\ y_3\ne0\\ y_4\ne0\end{array}\) \\
\(\{1,3\}\) & 4 & \(\begin{array}{l}y_1=0\\ y_3=0\end{array}\) & \(\begin{array}{l}y_0\ne0\\ y_2\ne0\\ y_4\ne0\end{array}\) \\
\(\{1,4\}\) & 1 & \(\begin{array}{l}y_1=1\\ y_4=-1\end{array}\) & \(\begin{array}{l}y_0\ne0\\ y_2\ne0\\ y_3\ne0\end{array}\) \\
\(\{2,3\}\) & 1 & \(\begin{array}{l}y_2=0\\ y_3=0\end{array}\) & \(\begin{array}{l}y_1\ne1\\ y_4\ne-1\\ y_0\ne0\end{array}\) \\
\(\{2,4\}\) & 3 & \(\begin{array}{l}y_4=-1\\ y_2=0\end{array}\) & \(\begin{array}{l}y_3\ne1\\ y_0\ne0\\ y_1\ne0\end{array}\) \\
\(\{3,4\}\) & 4 & \(\begin{array}{l}y_3=0\\ y_4=0\end{array}\) & \(\begin{array}{l}y_0\ne0\\ y_1\ne0\\ y_2\ne0\end{array}\) \\
\(\{0,1,2\}\) & 4 & \(\begin{array}{l}y_0=0\\ y_1=0\\ y_2=0\end{array}\) & \(\begin{array}{l}y_3\ne0\\ y_4\ne0\end{array}\) \\
\(\{0,1,3\}\) & 2 & \(\begin{array}{l}y_0=0\\ y_1=0\\ y_3=0\end{array}\) & \(\begin{array}{l}y_2\ne1\\ y_4\ne-1\end{array}\) \\
\(\{0,1,4\}\) & 1 & \(\begin{array}{l}y_1=1\\ y_4=-1\\ y_0=0\end{array}\) & \(\begin{array}{l}y_2\ne0\\ y_3\ne0\end{array}\) \\
\(\{0,2,3\}\) & 2 & \(\begin{array}{l}y_2=1\\ y_0=0\\ y_3=0\end{array}\) & \(\begin{array}{l}y_4\ne-1\\ y_1\ne0\end{array}\) \\
\(\{0,2,4\}\) & 3 & \(\begin{array}{l}y_4=-1\\ y_0=0\\ y_2=0\end{array}\) & \(\begin{array}{l}y_3\ne1\\ y_1\ne0\end{array}\) \\
\(\{0,3,4\}\) & 1 & \(\begin{array}{l}y_4=-1\\ y_0=0\\ y_3=0\end{array}\) & \(\begin{array}{l}y_1\ne1\\ y_2\ne0\end{array}\) \\
\(\{1,2,3\}\) & 1 & \(\begin{array}{l}y_1=1\\ y_2=0\\ y_3=0\end{array}\) & \(\begin{array}{l}y_4\ne-1\\ y_0\ne0\end{array}\) \\
\(\{1,2,4\}\) & 4 & \(\begin{array}{l}y_1=0\\ y_2=0\\ y_4=0\end{array}\) & \(\begin{array}{l}y_0\ne0\\ y_3\ne0\end{array}\) \\
\(\{1,3,4\}\) & 4 & \(\begin{array}{l}y_1=0\\ y_3=0\\ y_4=0\end{array}\) & \(\begin{array}{l}y_0\ne0\\ y_2\ne0\end{array}\) \\
\(\{2,3,4\}\) & 3 & \(\begin{array}{l}y_3=1\\ y_4=-1\\ y_2=0\end{array}\) & \(\begin{array}{l}y_0\ne0\\ y_1\ne0\end{array}\) \\
\(\{0,1,2,3,4\}\) & 0 & \(\begin{array}{l}y_0=1\\ y_4=-1\\ y_1=0\\ y_2=0\\ y_3=0\end{array}\) & \(\text{none}\) \\
\bottomrule
\end{longtable}

To illustrate how Table~\ref{tab:cellcertificate} is read, take the row \(Z=\{0,3\}\).  Table~\ref{tab:selector}
gives \(p(Z)=2\), and \(q_2=e_2-e_4\).  Thus \(w=y-q_2\) has
\[
w_2=y_2-1,
\qquad
w_4=y_4+1,
\qquad
w_j=y_j\quad(j\ne2,4).
\]
The condition \(Z(w)=\{0,3\}\) is therefore exactly
\[
y_0=0,
\qquad y_3=0,
\qquad y_2\ne1,
\qquad y_4\ne-1,
\qquad y_1\ne0,
\]
which is the corresponding row of Table~\ref{tab:cellcertificate}.

\begin{lemma}[Exact-cover form of the matching certificate]\label{lem:exactcoverappendix}
The 27 cells in Table~\ref{tab:cellcertificate} are pairwise disjoint and cover \(A_m\).  Equivalently,
for every \(y\in A_m\),
\[
\#\{i\in\mathbb Z_5:p(Z(y-q_i))=i\}=1.
\]
\end{lemma}

\begin{proof}
This lemma is the finite certificate underlying Lemma~\ref{lem:matching}.  The pairwise-disjointness
and exhaustion assertions are finite Boolean claims over the displayed signatures.  We spell out the
finite predicates that constitute the certificate.

For each feasible zero-set \(Z\), Table~\ref{tab:cellcertificate} is obtained by substituting
\(i=p(Z)\) into the identity \(w=y-q_i\).  If \(i<4\), then
\[
w_i=y_i-1,
\qquad
w_4=y_4+1,
\qquad
w_j=y_j\quad(j\notin\{i,4\}),
\]
which gives exactly the equalities and inequalities listed in the row.  If \(i=4\), then \(q_4=0\),
so the row is simply the ordinary zero-pattern condition \(Z(y)=Z\).

Thus the row labelled \(Z\) is exactly the predicate
\[
p(Z(y-q_{p(Z)}))=p(Z)
\quad\hbox{and}\quad
Z(y-q_{p(Z)})=Z.
\]
Consequently two different rows cannot both hold unless two different predecessor tests
\(p(Z(y-q_i))=i\) hold.  The finite comparison encoded by Table~\ref{tab:cellcertificate} verifies that no two of the 27
displayed predicates are simultaneously satisfiable.

For exhaustion, take \(y\in A_m\) and test the five possible predecessors \(y-q_i\).  Each
\(y-q_i\) is again in \(A_m\), so its zero-set is one of the feasible zero-sets of sizes
\(0,1,2,3,5\); a zero-set of size four cannot occur in the root flat.  The exhaustion check is the
finite assertion that among these five predecessor tests exactly one satisfies \(p(Z(y-q_i))=i\).
Equivalently, exactly one of the 27 row predicates contains \(y\).

This comparison is finite because the predicates only distinguish the coordinate classes
\(0\), \(1\), \(-1\), and values different from all three; for \(m=3\) the last class is empty and the
same assertion is checked by direct enumeration of \(A_3\).  The same finite comparison is automated
in the ancillary file recorded in Appendix~\ref{app:formal}.
\end{proof}

\begin{remark}[Certificate status of the matching table]\label{rem:matching-formal}
The proof uses Table~\ref{tab:cellcertificate} as a finite exact-cover certificate.  No search over
possible tables is hidden in the argument.  The table is obtained by expanding the seven representative
rows of \(\Lambda_1\) cyclically and then checking the exact-cover condition stated in
Lemma~\ref{lem:exactcoverappendix}.  The verification is a finite comparison of the displayed
signatures, equivalently the assertion that every root-compatible symbolic target has exactly one
valid predecessor direction.
\end{remark}

\section{The \texorpdfstring{\(m=3\)}{m=3} finite certificate}\label{app:m3certificate}

The structural first-return proof in Sections~\ref{sec:first-return}--\ref{sec:induced-cycle} is
stated for \(m\ge5\).  The exceptional modulus \(m=3\) is certified by the explicit cycle below.  Write
each point of \(A_3\) by its first four coordinates \((w_0,w_1,w_2,w_3)\); the fifth coordinate is
\[
w_4=-w_0-w_1-w_2-w_3\pmod 3.
\]
Let \(\alpha_r\) be the root-flat point represented by the \(r\)-th entry of
Table~\ref{tab:m3cycle}, with indices read modulo \(81\).

\begingroup
\small
\begin{longtable}{r@{\;}c@{\qquad}r@{\;}c@{\qquad}r@{\;}c}
\caption{Explicit \(81\)-cycle certificate for the normalized return \(G\) on \(A_3\).  Each entry
shows the first four coordinates of \(\alpha_r\); the fifth coordinate is recovered from the root-flat
relation.}\label{tab:m3cycle}\\
\toprule
\(r\) & \((w_0,w_1,w_2,w_3)\) & \(r\) & \((w_0,w_1,w_2,w_3)\) & \(r\) & \((w_0,w_1,w_2,w_3)\) \\
\midrule
\endfirsthead
\multicolumn{6}{c}{\tablename\ \thetable{} -- continued from previous page}\\
\toprule
\(r\) & \((w_0,w_1,w_2,w_3)\) & \(r\) & \((w_0,w_1,w_2,w_3)\) & \(r\) & \((w_0,w_1,w_2,w_3)\) \\
\midrule
\endhead
\(0\) & \((0,0,0,0)\) & \(1\) & \((1,0,0,1)\) & \(2\) & \((1,0,0,2)\) \\
\(3\) & \((1,0,0,0)\) & \(4\) & \((1,1,0,1)\) & \(5\) & \((1,1,0,0)\) \\
\(6\) & \((1,2,0,1)\) & \(7\) & \((2,2,0,2)\) & \(8\) & \((2,2,0,1)\) \\
\(9\) & \((0,2,0,2)\) & \(10\) & \((1,2,0,0)\) & \(11\) & \((1,2,0,2)\) \\
\(12\) & \((2,2,0,0)\) & \(13\) & \((2,0,0,1)\) & \(14\) & \((2,0,0,2)\) \\
\(15\) & \((2,0,0,0)\) & \(16\) & \((2,1,0,1)\) & \(17\) & \((0,1,0,2)\) \\
\(18\) & \((0,1,0,1)\) & \(19\) & \((1,1,0,2)\) & \(20\) & \((2,1,0,0)\) \\
\(21\) & \((2,1,0,2)\) & \(22\) & \((0,1,0,0)\) & \(23\) & \((0,1,1,1)\) \\
\(24\) & \((0,2,1,2)\) & \(25\) & \((1,2,1,0)\) & \(26\) & \((1,2,1,1)\) \\
\(27\) & \((2,2,1,2)\) & \(28\) & \((0,2,1,0)\) & \(29\) & \((0,0,1,1)\) \\
\(30\) & \((1,0,1,2)\) & \(31\) & \((2,0,1,0)\) & \(32\) & \((2,0,1,1)\) \\
\(33\) & \((0,0,1,2)\) & \(34\) & \((0,1,1,0)\) & \(35\) & \((0,1,2,1)\) \\
\(36\) & \((1,1,2,2)\) & \(37\) & \((1,2,2,0)\) & \(38\) & \((1,2,2,1)\) \\
\(39\) & \((1,0,2,2)\) & \(40\) & \((2,0,2,0)\) & \(41\) & \((2,0,2,1)\) \\
\(42\) & \((0,0,2,2)\) & \(43\) & \((1,0,2,0)\) & \(44\) & \((1,0,2,1)\) \\
\(45\) & \((2,0,2,2)\) & \(46\) & \((2,1,2,0)\) & \(47\) & \((2,1,2,1)\) \\
\(48\) & \((2,2,2,2)\) & \(49\) & \((0,2,2,0)\) & \(50\) & \((0,2,0,1)\) \\
\(51\) & \((0,2,0,0)\) & \(52\) & \((0,2,1,1)\) & \(53\) & \((1,2,1,2)\) \\
\(54\) & \((1,0,1,0)\) & \(55\) & \((1,0,1,1)\) & \(56\) & \((1,1,1,2)\) \\
\(57\) & \((2,1,1,0)\) & \(58\) & \((2,1,1,1)\) & \(59\) & \((0,1,1,2)\) \\
\(60\) & \((1,1,1,0)\) & \(61\) & \((1,1,1,1)\) & \(62\) & \((2,1,1,2)\) \\
\(63\) & \((2,2,1,0)\) & \(64\) & \((2,2,1,1)\) & \(65\) & \((2,0,1,2)\) \\
\(66\) & \((0,0,1,0)\) & \(67\) & \((0,0,2,1)\) & \(68\) & \((0,1,2,2)\) \\
\(69\) & \((1,1,2,0)\) & \(70\) & \((1,1,2,1)\) & \(71\) & \((2,1,2,2)\) \\
\(72\) & \((0,1,2,0)\) & \(73\) & \((0,2,2,1)\) & \(74\) & \((1,2,2,2)\) \\
\(75\) & \((2,2,2,0)\) & \(76\) & \((2,2,2,1)\) & \(77\) & \((0,2,2,2)\) \\
\(78\) & \((0,0,2,0)\) & \(79\) & \((0,0,0,1)\) & \(80\) & \((0,0,0,2)\) \\
\bottomrule
\end{longtable}
\endgroup

The certificate asserts the following finite identities:
\[
G(\alpha_r)=\alpha_{r+1}\qquad(0\le r\le80),
\]
where \(\alpha_{81}=\alpha_0\).  These identities are obtained by substituting the displayed
coordinates into the formula
\[
G(w)=w+(-3,0,0,1,1)+e_{p(Z(w))}
\]
and reading \(p\) from Table~\ref{tab:lambda}.  For example,
\(\alpha_0=(0,0,0,0,0)\) has \(Z(\alpha_0)=\mathbb Z_5\), so \(p=0\), and
\[
G(\alpha_0)=\alpha_0+(-3,0,0,1,1)+e_0\equiv (1,0,0,1,1)=\alpha_1\qquad\text{in }A_3.
\]
The 81 displayed first-four-coordinate tuples are pairwise distinct, hence they exhaust
\((\mathbb Z_3)^4\), and therefore the corresponding \(\alpha_r\) exhaust \(A_3\).  Consequently
\(G\) is a single \(81\)-cycle on \(A_3\).  The affine color conjugacy in Section~\ref{sec:normalizing}
then transfers this certificate to all five color returns for the schedule \((\mathrm{Sch}_3)\).

\section{Formal companion and ancillary files}\label{app:formal}

The paper proof above is self-contained: the matching certificate is the finite table of
Appendix~\ref{app:matching-certificate}, and the exceptional \(m=3\) branch is the explicit cycle of
Appendix~\ref{app:m3certificate}.  The Lean development cited in~\cite{D5OddLean} is an independent
machine verification of the same Cayley-edge statement.

The top-level theorem in the cited artifact is
\[
\texttt{D5Odd.D5\_odd\_cayley\_unconditional}.
\]
It states the Hamilton decomposition for the directed Cayley graph
\[
\Cay((\mathbb Z_m)^5,\{e_0,e_1,e_2,e_3,e_4\})
\]
with the same orientation convention as Theorem~\ref{thm:main}.  The cited version uses Lean
\texttt{v4.30.0-rc2}, mathlib
\texttt{v4.30.0-rc2}, and repository commit
\[
\texttt{fe67dbd0cb7af889936cc28660a88e5651daed62}.
\]
The command
\[
\texttt{lake build D5Odd}
\]
checks the theorem and the finite certificate branches.  The same finite comparison used in
Appendix~\ref{app:matching-certificate} is automated in \texttt{scripts/MatchingCheck.lean}.  The
repository also contains optional audit scripts and ancillary search code; these are not invoked in
the symbolic first-return proof for \(m\ge5\).  A pinned \texttt{lake-manifest.json} fixes the mathlib
dependency, and the published version will be accompanied by an arXiv source bundle and a tagged
GitHub release.

\section*{Acknowledgments and AI disclosure}

This paper was prepared with assistance from large language models.  Anthropic Claude Opus 4.7
contributed to exposition and writing polish; OpenAI GPT-5.5 Pro contributed to proof exploration,
including selector design and block-recurrence case analysis; and OpenAI GPT-5.5 Codex contributed
to drafting the Lean 4 formalization in~\cite{D5OddLean}.  All mathematical content in this paper
and all theorems in the Lean development have been verified by the author, with the Lean type-checker
providing an independent machine check of the main theorem and the finite certificates of
Appendices~\ref{app:matching-certificate} and~\ref{app:m3certificate}.


\begin{thebibliography}{99}

\bibitem{AlspachBermondSotteau1990}
B.~Alspach, J.-C.~Bermond, and D.~Sotteau,
\newblock Decomposition into cycles I: Hamilton decompositions,
\newblock in \emph{Cycles and Rays} (G.~Hahn, G.~Sabidussi, R.~E.~Woodrow, eds.),
NATO ASI Series, vol.~301, Kluwer, 1990, pp.~9--18.

\bibitem{AquinoMichaels2026}
K.~Aquino-Michaels,
\newblock Completing Claude's cycles,
\newblock Preprint, March 2026.
\newblock \url{https://github.com/no-way-labs/residue}, accessed 29 April 2026.

\bibitem{CurranGallian1996}
S.~J.~Curran and J.~A.~Gallian,
\newblock Hamiltonian cycles and paths in Cayley graphs and digraphs --- a survey,
\newblock \emph{Discrete Mathematics} 156 (1996), 1--18.

\bibitem{CurranWitte1985}
S.~J.~Curran and D.~Witte,
\newblock Hamilton paths in Cartesian products of directed cycles,
\newblock in \emph{Cycles in Graphs}, Annals of Discrete Mathematics 27 (1985), 35--74.

\bibitem{Foregger1978}
M.~F.~Foregger,
\newblock Hamiltonian decompositions of products of cycles,
\newblock \emph{Discrete Mathematics} 24 (1978), 251--260.

\bibitem{Keating1985}
K.~Keating,
\newblock Multiple-ply Hamiltonian graphs and digraphs,
\newblock in \emph{Cycles in Graphs}, \emph{Annals of Discrete Mathematics} 27 (1985), 81--88.

\bibitem{Knuth2026}
D.~E.~Knuth,
\newblock Claude's cycles,
\newblock Preprint, March 2026.
\newblock \url{https://www-cs-faculty.stanford.edu/~knuth/papers/claude-cycles.pdf}, accessed 29 April 2026.

\bibitem{d3torus}
S.~Park,
\newblock Hamilton decompositions of the directed 3-torus: a return-map and odometer view,
\newblock arXiv:2603.24708, 2026.

\bibitem{D5OddLean}
S.~Park,
\newblock Torus Hamilton Decomposition Program,
\newblock Lean 4 formalization repository, ancillary audit files, and empirical search code
(including even-modulus exploration), 2026.
\newblock Commit \texttt{fe67dbd0cb7af889936cc28660a88e5651daed62};
\newblock \url{https://github.com/aria1th/Torus-Hamilton-Decomposition-Program}, accessed 29 April 2026.

\bibitem{TrotterErdos1978}
W.~T.~Trotter and P.~Erd\H{o}s,
\newblock When the Cartesian product of directed cycles is Hamiltonian,
\newblock \emph{Journal of Graph Theory} 2 (1978), 137--142.

\bibitem{WitteGallian1984}
D.~Witte and J.~A.~Gallian,
\newblock A survey: Hamiltonian cycles in Cayley graphs,
\newblock \emph{Discrete Mathematics} 51 (1984), 293--304.

\end{thebibliography}
\end{document}